\documentclass[12pt]{article}
\usepackage{amssymb}
\usepackage{mathrsfs}
\usepackage[hypertex]{hyperref}
\usepackage{amsfonts}
\usepackage{graphicx}
\usepackage{mathptmx}
\usepackage{latexsym,amsmath,amssymb,amsfonts,amsthm}

\newtheorem{theorem}{Theorem}[section]
\newtheorem{lemma}[theorem]{Lemma}

\newtheorem{remark}[theorem]{Remark}

\begin{document}
\setcounter{page}{1}
\title{Existence and uniqueness of solutions to the constant mean curvature equation with nonzero Neumann boundary data in product manifold $M^{n}\times\mathbb{R}$}
\author{Ya Gao,~~Jing Mao$^{\ast}$,~~Chun-Lan Song}
\date{}
\protect\footnotetext{\!\!\!\!\!\!\!\!\!\!\!\!{$^{\ast}$Corresponding
author}
}
\maketitle ~~~\\[-15mm]
\begin{center}{\footnotesize
Faculty of Mathematics and Statistics, \\
Key Laboratory of Applied Mathematics of Hubei Province, \\
Hubei University, Wuhan 430062, China\\
Emails: jiner120@163.com, jiner120@tom.com }
\end{center}

\begin{abstract}
In this paper, we can prove the existence and uniqueness of
solutions to the constant mean curvature (CMC for short) equation
with nonzero Neumann boundary data in product manifold
$M^{n}\times\mathbb{R}$, where $M^{n}$ is an $n$-dimensional
($n\geq2$) complete Riemannian manifold with nonnegative Ricci
curvature, and $\mathbb{R}$ is the Euclidean $1$-space.
Equivalently, this conclusion gives the existence of CMC graphic
hypersurfaces defined over a compact strictly convex domain
$\Omega\subset M^{n}$ and having arbitrary contact angle.
\end{abstract}

\markright{\sl\hfill  Y. Gao, J. Mao, C.-L. Song\hfill}

\section{Introduction}
\renewcommand{\thesection}{\arabic{section}}
\renewcommand{\theequation}{\thesection.\arabic{equation}}
\setcounter{equation}{0} \setcounter{maintheorem}{0}

Recent years, the study of submanifolds of constant curvature in
product manifolds attracts many geometers' attention. For instance,
Hopf in 1955 discovered that the complexification of the traceless
part of the second fundamental form of an immersed surface
$\Sigma^{2}$, with constant mean curvature $H$, in $\mathbb{R}^3$ is
a holomorphic quadratic differential $Q$ on $\Sigma^{2}$, and then
he used this observation to get his well-known conclusion that any
immersed CMC sphere $\mathbb{S}^{2}\hookrightarrow\mathbb{R}^3$ is a
standard distance sphere with radius $1/H$. By introducing a
generalized quadratic differential $\widetilde{Q}$ for immersed
surfaces $\Sigma^{2}$ in product spaces
$\mathbb{S}^{2}\times\mathbb{R}$ and
$\mathbb{H}^{2}\times\mathbb{R}$, with $\mathbb{S}^{2}$,
$\mathbb{H}^{2}$ the $2$-dimensional sphere and hyperbolic surface
respectively, Abresch and Rosenberg \cite{ar} can extend Hopf's
result to CMC spheres in these target spaces. Meeks and Rosenberg
\cite{mr} successfully classified stable properly embedded
orientable minimal surfaces in the product space
$M\times\mathbb{R}$, where $M$ is a closed orientable Riemannian
surface. In fact, they proved that such a surface must be a product
of a stable embedded geodesic on $M$ with $\mathbb{R}$, a minimal
graph over a region of $M$ bounded by stable geodesics,
$M\times\{t\}$ for some $t\in\mathbb{R}$, or is in a moduli space of
periodic multigraphs parameterized by $P\times\mathbb{R}^{+}$, where
$P$ is the set of primitive (non-multiple) homology classes in
$H_{1}(M)$. Mazet, Rodr\'{\i}guez and Rosenberg \cite{lmh} analyzed
properties of periodic minimal or CMC surfaces in the product
manifold $\mathbb{H}^{2}\times\mathbb{R}$, and they also construct
examples of periodic minimal surfaces in
$\mathbb{H}^{2}\times\mathbb{R}$. In \cite{hfj}, Rosenberg, Schulze
and Spruck showed that a properly immersed minimal hypersurface in
$M\times\mathbb{R}^{+}$ equals some slice $M\times\{c\}$ when $M$ is
a complete, recurrent $n$-dimensional Riemannian manifold with
bounded curvature. Of course, for more information, readers can
check references therein of these papers.\footnote{ In fact, readers
can check \cite[Remark 1.2 (III)]{ggm} for this content also.
However, we prefer to write it down here again to \emph{clearly}
expose our motivation of investigating CMC hypersurfaces in product
manifold $M^{n}\times\mathbb{R}$.} Hence, it is interesting and
important to consider submanifolds of constant curvature in the
product manifold of type $M^{n}\times\mathbb{R}$.

Let $(M^{n},\sigma)$ be a complete $n$-manifold ($n\geq2$) with the
Riemannian metric $\sigma$, and let $\Omega\subset M^{n}$ be a
compact strictly convex domain with smooth boundary
$\partial\Omega$. Denote by
$\left(U_{A};w^{1}_{A},w^{2}_{A},\cdots,w^{n}_{A}\right)$ the local
coordinate coverings of $M$, and $\frac{\partial}{\partial
w^{i}_{A}}$, $i=1,2,\cdots,n$, the corresponding coordinate vector
fields, where $A\in I\subseteq N$ with $N$ the set of all positive
integers. For simplicity, we just write
$\{w^{1}_{A},w^{2}_{A},\cdots,w^{n}_{A}\}$ as
$\{w^{1},w^{2},\cdots,w^{n}\}$ to represent the local coordinates on
$M$, and write $\frac{\partial}{\partial w^{i}_{A}}$ as
$\frac{\partial}{\partial w^{i}}$ or $\partial_{i}$. In this
setting, the metric $\sigma$ should be
$\sigma=\sum_{i,j=1}^{n}\sigma_{ij}dw^{i}\otimes dw^{j}$ with
$\sigma_{ij}=\sigma(\partial_{i},\partial_{j})$. Denote by $D$,
$D^{\partial\Omega}$ the covariant derivatives on $\Omega$ and
$\partial\Omega$ respectively.  Given a smooth\footnote{ In fact, it
is not necessary to impose smoothness assumption on the initial
hypersurface $\mathcal{G}$. The $C^{2,\alpha}$-regularity for
$\mathcal{G}$ is enough to get all the estimates in the sequel.
However, in order to avoid the boring regularity arguments, which is
not necessary, here we assume $\mathcal{G}$ is smooth.} graphic
hypersurface $\mathcal{G}\subset M^{n}\times\mathbb{R}$ defined over
$\Omega$, where $M^{n}\times\mathbb{R}$ is the product manifold with
the product metric $\overline{g}=\sigma_{ij}dw^{i}\otimes
dw^{j}+ds\otimes ds$, then there exists a smooth function $u_{0}\in
C^{\infty}(\overline{\Omega})$ such that $\mathcal{G}$ can be
represented by $\mathcal{G}:=\{(x,u_{0}(x))|x\in\Omega\}$. It is not
hard to know that the metric of  $\mathcal{G}$ is given by
$g=i^{\ast}\overline{g}$, where $i^{\ast}$ is the pullback mapping
of the immersion $i:\mathcal{G}\hookrightarrow
M^{n}\times\mathbb{R}$, tangent vectors are given by
\begin{eqnarray*}
\vec{e_{i}}=\partial_{i}+D_{i}u\partial_{s}, \qquad  i=1,2,\cdots,n,
\end{eqnarray*}
 and the corresponding upward unit normal vector is given by
\begin{eqnarray*}
\vec{\gamma}=-\frac{\sum\limits_{i=1}^{n}D^{i}u\partial_{i}-\partial_s}{\sqrt{1+|Du|^2}},
\end{eqnarray*}
where $D^{j}u=\sum_{i=1}^{n}\sigma^{ij}D_{i}u$. Denote by $\nabla$
the covariant derivative operator on $M^{n}\times\mathbb{R}$, and
then the second fundamental form $h_{ij}d\omega^{i}\otimes
d\omega^{j}$ of $\mathcal{G}$ is given by
\begin{eqnarray*}
h_{ij}=-\langle\nabla_{\vec{e}_i}\vec{e}_j,\vec{\gamma}\rangle_{\overline{g}}=-\frac{D_{i}D_{j}u}{\sqrt{1+|Du|^2}}.
\end{eqnarray*}
Moreover, the scalar mean curvature of $\mathcal{G}$ is
\begin{eqnarray} \label{smc}
\qquad
H=\sum_{i=1}^{n}h^i_i=-\frac{\sum\limits_{i,k=1}^{n}g^{ik}D_{i}D_{k}u}{\sqrt{1+|Du|^2}}=-\frac{\sum\limits_{i,k=1}^{n}\left(\sigma^{ik}-\frac{D^{i}uD^{k}u}{1+|Du|^{2}}\right)D_{i}D_{k}u}{\sqrt{1+|Du|^2}}.
\end{eqnarray}
Our purpose is to consider the following CMC equation with nonzero
Neumann boundary condition (NBC for short)
\begin{eqnarray*}
(\natural)\qquad \left\{
\begin{array}{ll}
H=\mathrm{div}\left(\frac{Du}{\sqrt{1+|Du|^{2}}}\right)=\lambda \qquad \qquad &\mathrm{in} ~ \Omega,\\
D_{\vec{\nu}}u=\phi(x) \qquad \qquad &\mathrm{on} ~
\partial \Omega,
\end{array}
\right.
\end{eqnarray*}
with $\lambda\in\mathbb{R}$ a constant, and try to get the existence
of solutions to ($\natural$). Here $\vec{\nu}$ is the inward unit
normal vector of $\partial\Omega$ and $\phi(x)\in
C^{\infty}(\overline{\Omega})$. Inspired by the method for the
gradient estimate \cite[Lemma 2.2]{ggm}, we can successfully finish
this purpose. In fact, we can prove:

\begin{theorem} \label{maintheorem}
If the Ricci curvature of $M^{n}$ is nonnegative, then there exist a
unique $\lambda\in R$ and a function $u\in C^{\infty}
(\overline{\Omega})$ solving ($\natural$). Moreover, the solution
$u$ is unique up to a constant.

\end{theorem}

\begin{remark}
\rm{ (I) The cosine of the contact angle between $\vec{\gamma}$ and
$\vec{\nu}$ is
\begin{eqnarray*}
 \langle\vec{\gamma},\vec{\nu}\rangle_{\overline{g}}=\frac{D_{\vec{\nu}}u}{\sqrt{1+|Du|^2}}.
\end{eqnarray*}
Hence, if the contact angle is arbitrary, then there should exist
some $\varphi(x)\in
  C^{\infty}(\overline{\Omega})$, $|\varphi(x)|\leq1$ on
  $\partial\Omega$ such that
  $D_{\vec{\nu}}u\big{|}_{\partial\Omega}=\varphi(x)\cdot\sqrt{1+|Du|^2}$. Based on this reason, we can say that although
  the boundary value problem\footnote{ We write it as BVP for short.} ($\natural$) has nonzero NBC, the geometric meaning of the
  NBC in ($\natural$) is not sufficient. \emph{Can we deal with the
  BVP ($\natural$) if the RHS of the nonzero NBC therein contains $Du$
  also?} Inspired by a recent work \cite{wwx}, Gao and Mao
  \cite{gm2} considered a generalization of the BVP ($\natural$) where the NBC can be replaced by
  \begin{eqnarray*}
D_{\vec{\nu}}u=\phi(x)\cdot\left(\sqrt{1+|Du|^2}\right)^{\frac{1-q}{2}}
  \end{eqnarray*}
for any $q>0$, and similar conclusion to Theorem \ref{maintheorem}
could be derived.
\\
(II) Clearly, the solvability of ($\natural$) implies the existence
of CMC graphic hypersurfaces defined over $\Omega\subset M^{n}$ and
having arbitrary contact angle.\\
(III) Clearly, if $M^{n}\equiv\mathbb{R}^{n}$, our main conclusion
here becomes \cite[Theorem 1.3]{mww} exactly. That is to say,
Theorem \ref{maintheorem} covers \cite[Theorem 1.3]{mww} as a
special case.}
\end{remark}

\section{Proof of Theorem \ref{maintheorem}}
\renewcommand{\thesection}{\arabic{section}}
\renewcommand{\theequation}{\thesection.\arabic{equation}}
\setcounter{equation}{0} \setcounter{maintheorem}{0}

We know that if $\Omega$ is a strictly convex domain with smooth
boundary $\partial\Omega$, then there exists a smooth function
$\beta$ on $\Omega$ such that $\beta|_{\Omega}<0$,
$\beta|_{\partial\Omega}=0$, ${\sup_\Omega}|D\beta|\leq1$,
 \begin{eqnarray*}
 \left(\beta_{ij}\right)_{n\times n}\geq
k_{0}\left(\delta_{ij}\right)_{n\times n}
 \end{eqnarray*}
 for some positive constant $k_{0}>0$,
$\beta_{\vec{\nu}}=D_{\vec{\nu}}\beta=-1$ and $|D\beta|=1$ on
$\partial\Omega$. Besides, since $\Omega$ is strictly convex, we
have
 \begin{eqnarray*}
\left(h_{ij}^{\partial\Omega}\right)_{(n-1)\times(n-1)}\geq
\kappa_{1}\left(\delta_{ij}\right)_{(n-1)\times(n-1)},
 \end{eqnarray*}
where $h_{ij}^{\partial\Omega}$, $1\leq i,j\leq n-1$, is the second
fundamental form of the boundary $\partial\Omega$, and
$\kappa_{1}>0$ is the minimal principal curvature of
$\partial\Omega$.

\begin{lemma} \label{lemma2.1}
Let $\varepsilon>0$ and $\phi\in C^{3}(\overline{\Omega})$. Assume
that there exists a positive constant $L$ such that
\begin{eqnarray*}
           |\phi|_{C^{3}(\overline{\Omega})}\leq L,
\end{eqnarray*}
and the Ricci curvature of $M^{n}$ is nonnegative. Let $u$ be the
solution to BVP ($\natural$). Then there exists a
constant~$c_{1}=c_{1}(n,\Omega,L)$~such that
 \begin{eqnarray*}
 \sup_{\overline{\Omega}}|Du|\leq c_{1}.
 \end{eqnarray*}
\end{lemma}

\begin{proof}
We use a similar method to that of the proof of \cite[Lemma
2.2]{ggm}.

Denote by
$a^{ij}:=(1+|Du|^{2})\sigma^{ij}-D^{i}uD^{j}u$,~$f=\varepsilon u$,
$v=\sqrt{1+|Du|^{2}}$. Then first equation in ($\natural$) can be
rewritten as
\begin{eqnarray*}
\sum\limits_{i,j=1}^{n}a^{ij}u_{ij}=fv^{3}.
\end{eqnarray*}

Let
 \begin{eqnarray*}
\Phi=\log|D\omega|^{2}+\zeta\beta,
 \end{eqnarray*}
where~$\omega=u+\phi(x)\beta$, and $\zeta$~is a positive constant
determined later. For convenience, denoted by $G=-\phi(x)\beta$.

We first show the maximum of $\Phi(x)$ on $\overline{\Omega}$~cannot
be achieved at the boundary $\partial\Omega$. This fact can be shown
by using the same argument as (2.1) in \cite{ggm}. However, for
completeness, we would like to repeat here.

Choose a suitable local coordinates around a point
$x_{0}\in\overline{\Omega}$ such that $\tau_{n}$ is the inward unit
normal vector of $\partial\Omega$, and $\tau_{i}$,
$i=1,2,\cdots,n-1$, are the unit smooth tangent vectors of
$\partial\Omega$. Denote by $D_{\tau_{i}}u:=u_{i}$,
$D_{\tau_{j}}u:=u_{j}$, $D_{i}D_{j}u:=u_{ij}$ for $1\leq i,j \leq
n$. By the boundary condition, one has
 \begin{eqnarray*}
D_{\tau_{n}}\omega\big{|}_{\partial\Omega}=\omega_{n}\big{|}_{\partial\Omega}=u_{n}\big{|}_{\partial\Omega}+\left(\phi_{n}\beta+\beta_{n}\phi\right)\big{|}_{\partial\Omega}=0.
 \end{eqnarray*}
  If $\Phi(x,t)$ attains its maximum at
$(x_{0},t_{0})\in\partial\Omega$, then at $x_{0}$, we have
\begin{eqnarray} \label{add-1}
0\geq\Phi_{n}&=&\frac{|D\omega|^{2}_{n}}{|D\omega|^{2}}-\zeta =
\sum\limits_{k=1}^{n-1}\frac{2\omega^{k}D_{\tau_{n}}D_{\tau_{k}}\omega}{|D\omega|^{2}}-\zeta\nonumber\\
&=&
\sum\limits_{k=1}^{n-1}\frac{2\omega^{k}[\tau_{k}(\tau_{n}(\omega))-(D_{\tau_{k}}\tau_{n})\omega]}{|D\omega|^{2}}-\zeta\nonumber\\
&=&
-\sum\limits_{k=1}^{n-1}\frac{2\omega^{k}(D_{\tau_{k}}\tau_{n})(\omega)}{|D\omega|^{2}}-\zeta\nonumber\\
&=&
-\sum\limits_{k=1}^{n-1}\frac{2\omega^{k}\omega_{j}\langle D_{\tau_{k}}\tau_{n},\tau_{j}\rangle_{\sigma}}{|D\omega|^{2}}-\zeta\nonumber\\
&=&
\sum\limits_{k=1}^{n-1}\frac{2\omega^{k}\omega_{j}\langle D_{\tau_{k}}\tau_{j},\tau_{n}\rangle_{\sigma}}{|D\omega|^{2}}-\zeta\nonumber\\
&=&
\sum\limits_{k,j=1}^{n-1}\frac{2\omega^{k}\omega_{j}h_{kj}^{\partial\Omega}}{|D\omega|^{2}}-\zeta\nonumber\\
&\geq& 2\kappa_{1}-\zeta.
\end{eqnarray}
Hence, by taking $0<\zeta<2\kappa_{1}$, the maximum of $\Phi$ can
only be achieved in $\Omega$. BTW, there is one thing we would like
to mention here, that is , in (\ref{add-1}), the relation
 \begin{eqnarray*}
w^{k}=\sum\limits_{l=1}^{n}\sigma^{kl}w_{l}=\sum\limits_{l=1}^{n-1}\sigma^{kl}w_{l}
\end{eqnarray*}
holds. Here we have used the convention in Riemannian Geometry to
deal with the subscripts and superscripts, and this convention will
also be  used in the sequel.

 Assume $\Phi(x)$ attains its maximum at some point
 $x_{0}\in\Omega$. At $x_0$, as explained in the proof of \cite[Lemma
2.2]{ggm}, we can make a suitable change to the coordinate vector
fields $\{\tau_{1},\tau_{2},\cdots,\tau_{n}\}$ such that
$|Du|=u_{1}$, $(u_{ij})_{2\leq i,j\leq n}$ is diagonal, and
$(\sigma_{ij})_{2\leq i,j\leq n}$ is diagonal. Clearly, in this
setting, $\sigma^{11}=1$. Besides, we  have
 \begin{eqnarray*}
g^{11}&=&\frac{1}{v^{2}}, \quad g^{ij}=0~\mathrm{for} ~2\leq i,j\leq
n,i\neq j,~\mathrm{and}~g^{ii}=\sigma^{ii}~\mathrm{for} ~i\geq2,
\end{eqnarray*}
with $v=\sqrt{1+|Du|^{2}}=\sqrt{1+u^{2}_{1}}$, which leads to a fact
that, under this suitable frame field,
 \begin{eqnarray*}
a^{11}=1,\qquad a^{ii}=v^{2}\sigma^{ii}~~~~i=2,\ldots,n.
\end{eqnarray*}
Assume that $u_{1}$ is big enough such that $u_{1}$, $\omega_{1}$,
$\omega^{1}$, $|D\omega|$, and $v$ are equivalent with each other at
$x_{0}$. Otherwise Lemma \ref{lemma2.1} is proved. At $x_{0}$, we
have
\begin{eqnarray} \label{e-1}
\Phi_{i}=\frac{|D\omega|^{2}_{i}}{|D\omega|^{2}}+\zeta\beta_{i}=0,
\end{eqnarray}
and
 \begin{eqnarray*}
\Phi_{ij}=\frac{|D\omega|^{2}_{ij}}{|D\omega|^{2}}-
\frac{|D\omega|^{2}_{i}|D\omega|^{2}_{j}}{|D\omega|^{4}}
+\zeta\beta_{ij},
 \end{eqnarray*}
which implies
\begin{eqnarray}\label{e-2}
0&\geq&\sum\limits_{i,j=1}^{n}a^{ij}\Phi_{ij}\nonumber\\
&=&
\sum\limits_{i,j=1}^{n}\frac{a^{ij}|D\omega|^{2}_{ij}}{|D\omega|^{2}}-\zeta^{2}\sum\limits_{i,j=1}^{n}a^{ij}\beta_{i}\beta_{j}+\zeta \sum\limits_{i,j=1}^{n}a^{ij}\beta_{ij}\nonumber\\
&\triangleq& I+II+III.
\end{eqnarray}
By (\ref{e-1}), for $i=1,2,\ldots,n$, one has
\begin{eqnarray} \label{e-3}
\sum\limits_{k=1}^{n}\omega^{k}u_{ki}=\sum\limits_{k=1}^{n}\omega^{k}\omega_{ki}+\sum\limits_{k=1}^{n}\omega^{k}G_{ki}=-\frac{\zeta
\beta_{i}v^{2}}{2}+O(v),
\end{eqnarray}
which implies
\begin{eqnarray*}
\omega^{1}u_{1i}+\omega^{i}u_{ii}=-\frac{\zeta\beta_{i}v^{2}}{2}+O(v)
\end{eqnarray*}
holds for $i=2,3,\ldots,n$. Hence, for $i=2,3,\ldots,n$, we have
\begin{eqnarray} \label{e-4}
u_{1i}=O(1)-\frac{\zeta\beta_{i}v}{2}-\frac{\omega^{i}}{v}u_{ii},
\end{eqnarray}
and specially, for $i=1$, it gives
\begin{eqnarray*}
\omega^{1}u_{11}+\sum\limits_{k=2}^{n}\omega^{k}u_{k1}=O(v)-\frac{\zeta\beta_{1}v^{2}}{2}.
\end{eqnarray*}
By (\ref{e-4}), it is not hard to get
\begin{eqnarray} \label{add-2}
u_{11}&=&O(1)-\frac{\zeta\beta_{1}v}{2}-\sum\limits_{k=2}^{n}\frac{\omega^{k}}{v}\left(O(1)-\frac{\zeta\beta_{k}v}{2}-\frac{\omega^{k}}{v}u_{kk}\right)\nonumber\\
&=&
O(1)-\frac{\zeta\beta_{1}v}{2}+\sum\limits_{k=2}^{n}\left(\frac{\omega^{k}}{v}\right)^{2}u_{kk},
\end{eqnarray}
which, together with the equation
$u_{11}+(1+u_{1}^{2})\sum\limits_{k=2}^{n}\sigma^{kk}u_{kk}=fv^{3}$,
leads to the following facts:
\begin{eqnarray} \label{e-5}
u_{11}+\sum\limits_{i=2}^{n}\sigma^{ii}u_{ii}=fv+\frac{u_{1}^{2}}{v^{2}}u_{11}=fv+O(1)-\frac{u_{1}^{2}\zeta\beta_{1}}{2v
}+\sum\limits_{k=2}^{n}\frac{(u_{1}\omega^{k})^{2}}{v^{4}}u_{kk}
\end{eqnarray}
and
\begin{eqnarray} \label{e-6}
fv=\frac{u_{11}}{v^{2}}+\sum\limits_{i=2}^{n}\sigma^{ii}u_{ii}=O(\frac{1}{v^{2}})-\frac{\zeta\beta_{1}}{2v}+\sum\limits_{k=2}^{n}\left[\sigma^{kk}
+\frac{(\omega^{k})^{2}}{v^{4}}\right]u_{kk}.
\end{eqnarray}
Now, we are going to estimate (\ref{e-2}). First, by direct
calculation, we can get
\begin{eqnarray} \label{e-7}
II=-\zeta^{2}\sum\limits_{i,j=1}^{n}a^{ij}\beta_{i}\beta_{j}=-\zeta^{2}\left(\beta^{2}_{1}+v^{2}\sum\limits_{i=2}^{n}\sigma^{ii}\beta^{2}_{i}\right)
\end{eqnarray}
and
\begin{eqnarray} \label{e-8}
III=\zeta\sum\limits_{i,j=1}^{n}a^{ij}\beta_{ij}\geq\zeta
k_{0}\left(1+(n-1)\eta+(n-1)\eta u^{2}_{1}\right),
\end{eqnarray}
where $\eta:=\min\{\sigma^{22},\sigma^{33},\ldots,\sigma^{nn}\}$.

Second, we need to estimate the term $I$. By direct computation, we
have
\begin{eqnarray}\label{e-9}
\sum\limits_{i,j=1}^{n}a^{ij}|D\omega|^{2}_{ij}&=&
2\sum\limits_{i,j,k=1}^{n}a^{ij}\omega^{k}u_{kij}-2\sum\limits_{i,j,k=1}^{n}a^{ij}\omega^{k}G_{kij}+2\sum\limits_{i,j,k=1}^{n}a^{ij}\sigma^{kk}u_{ki}u_{kj}\nonumber\\
&&\qquad-
4\sum\limits_{i,j,k=1}^{n}a^{ij}\sigma^{kk}u_{ki}G_{kj}+2\sum\limits_{i,j,k=1}^{n}a^{ij}\sigma^{kk}G_{ki}G_{kj}\nonumber\\
&:=& I_{1}+I_{2}+I_{3}+I_{4}+I_{5}.
\end{eqnarray}
For term $I_1$, using the Ricci identity, one has
\begin{eqnarray}\label{e-10}
I_{1}&=&2\sum\limits_{i,j,k=1}^{n}a^{ij}\omega^{k}u_{kij}\nonumber\\
&=&
2\sum\limits_{i,j,k=1}^{n}a^{ij}\omega^{k}( u_{ijk}+R^{l} _{ikj}u_{l})\nonumber\\
&=&
2\sum\limits_{i,j,k=1}^{n}\omega^{k}[(fv^{3})_{k}-(a^{ij})_{k}u_{ij}]+2\sum\limits_{i,j,k=1}^{n}a^{ij}\omega^{k}R^{l}_{ikj}u_{l}\nonumber\\
&=&
2\sum\limits_{k=1}^{n}(\varepsilon u_{k}\omega^{k}v^{3}+3fv\sum\limits_{l=1}^{n}u^{l}u_{lk}\omega^{k})-4\sum\limits_{i,l,k=1}^{n}\sigma^{ii}u_{ii}u^{l}u_{lk}\omega^{k}\nonumber\\
&&\quad+
4\sum\limits_{i,j,k,l=1}^{n}\sigma^{il}\sigma^{1j}u_{1}u_{lk}u_{ij}\omega^{k}+2\sum\limits_{i,j,k,l=1}^{n}a^{ij}\omega^{k}R^{l}_{ikj}u_{l}\nonumber \\
&\geq&
 \left[6fv-4(u_{11}+\sum\limits_{i=2}^{n}\sigma^{ii}u_{ii})\right]\sum\limits_{k,l=1}^{n}u^{l}u_{lk}\omega^{k}\nonumber\\
 &&\quad+
 4u_{1}\sum\limits_{i,k,l=1}^{n}\sigma^{il}u_{lk}u_{1i}\omega^{k}+2\sum\limits_{i,j,k=1}^{n}a^{ij}\omega^{k}R^{1}_{ikj}u_{1}\nonumber\\
&\triangleq& I_{11}+I_{12}+I_{13},
\end{eqnarray}
where $R^{l}_{ikj}$, $1\leq i,j,k,l\leq n$, are coefficients of the
curvature tensor on $M^n$.

As in the proof of \cite[Lemma 3.1]{ggm}, one knows that
$|f|=|\varepsilon u|\leq c_{2}(n,\Omega)$ for some nonnegative
constant $c_2$ depending only on $n$ and the domain $\Omega$ itself.
Therefore, for the term $I_{11}$, using (\ref{e-3}), (\ref{e-5}) and
(\ref{e-6}), we can obtain
\begin{eqnarray}\label{e-11}
I_{11}&=&\left[6fv-4(u_{11}+\sum\limits_{i=2}^{n}\sigma^{ii}u_{ii})\right]u_{1}\sum\limits_{k=1}^{n}u_{1k}\omega^{k}\nonumber\\
&=&
u_{1}\left[6fv-4\left(fv+O(1)-\frac{u_{1}^{2}\zeta\beta_{1}}{2v}+\sum\limits_{k=2}^{n}\frac{(u_{1}\omega^{k})^{2}}{v^{4}}u_{kk}\right)\right]
\cdot\left(-\frac{\zeta\beta_{1}v^{2}}{2}+O(v)\right)\nonumber\\
&=&
u_{1}\left[2fv+O(1)+\frac{2u_{1}^{2}\zeta\beta_{1}}{v}-4\sum\limits_{k=2}^{n}\frac{(u_{1}\omega^{k})^{2}}{v^{4}}u_{kk}\right]\cdot\left
(-\frac{\zeta\beta_{1}v^{2}}{2}+O(v)\right)\nonumber\\
&=&
-fu_{1}\zeta\beta_{1}v^{3}+O(v^{3})-\zeta^{2}\beta^{2}_{1}u^{3}_{1}v+\sum\limits_{k=2}^{n}O(v)u_{kk}\nonumber\\
&=&
\left[O(\frac{1}{v^{2}})+\frac{\zeta\beta_{1}}{2v}-\sum\limits_{k=2}^{n}\left(\sigma^{kk}+\frac{(\omega^{k})^{2}}{v^{4}}\right)u_{kk}\right]
u_{1}\zeta\beta_{1}v^{2}+O(v^{3})\nonumber\\
&&\quad -
\zeta^{2}\beta^{2}_{1}u^{3}_{1}v+\sum^{n}_{k=2}O(v)u_{kk}\nonumber\\
&=&
O(v^{3})-\zeta^{2}\beta^{2}_{1}u_{1}^{3}v+\sum\limits_{k=2}^{n}\left(O(v)-\sigma^{kk}u_{1}\zeta\beta_{1}v^{2}\right)u_{kk}.
\end{eqnarray}
For the term $I_{12}$, applying (\ref{e-3}),  (\ref{e-4}) and
(\ref{add-2}), we have
\begin{eqnarray}\label{e-12}
I_{12}&=&4u_{1}\sum\limits_{i,k,l=2}^{n}\sigma^{il}u_{lk}u_{1i}\omega^{k}\nonumber\\
&=&
4u_{1}\sum\limits_{k=1}^{n}u_{1k}u_{11}\omega^{k}+4u_{1}\sum\limits_{k=1}^{n}\sum\limits_{i=2}^{n}\sigma^{ii}u_{ki}u_{1i}\omega^{k}\nonumber\\
&=&
4u_{1}\left(O(1)-\frac{\zeta\beta_{1}v}{2}+\sum\limits_{k=2}^{n}(\frac{\omega^{k}}{v})^{2}u_{kk}\right)\cdot\left(-\frac{\zeta\beta_{1}v^{2}}{2}+O(v)\right)\nonumber\\
&&\quad+
4u_{1}\sum\limits_{i=2}^{n}\sigma^{ii}\left(O(1)-\frac{\zeta\beta_{1}v}{2}-\frac{\omega^{i}}{v}u_{ii}\right)\cdot\left(-\frac{\zeta\beta_{1}v^{2}}{2}+O(v)\right)\nonumber\\
&=&
O(v^{3})+u_{1}\zeta^{2}\beta^{2}_{1}v^{3}+\sum\limits_{i=2}^{n}u_{1}\zeta^{2}\beta^{2}_{i}v^{3}\sigma^{ii}+\sum\limits_{i=2}^{n}O(v^{2})u_{ii}.
\end{eqnarray}
Moreover, for the term $I_{13}$, it can be deduced that
\begin{eqnarray}\label{e-13}
I_{13}&=&2\sum\limits_{i,j,k=1}^{n}g^{ij}\omega^{k}R^{1}_{ikj}u_{1}\nonumber\\
&=&
2\sum\limits_{k=1}^{n}\omega^{k}R^{1}_{1k1}u_{1}+2\sum\limits_{k=1}^{n}\sum\limits_{i=2}^{n}v^{2}\sigma^{ii}\omega^{k}R^{1}_{iki}u_{1}\nonumber\\
&=&
2\sum\limits_{k=2}^{n}\sigma^{kk}\omega_{k}R^{1}_{1k1}v+\sum\limits_{i=2}^{n}2v^{4}\sigma^{ii} R^{1}_{i1i}+2\sum\limits_{i,k=2}^{n}\sigma^{ii}\sigma_{kk}\omega_{k}v^{3} R^{1}_{iki}\nonumber\\
&\geq& O(v^{3}),
\end{eqnarray}
where the last inequality holds because of the nonnegativity of the
Ricci curvature on $M^n$, and the usage of the assumption that
$u_{1}$, $\omega_{1}$, $\omega^{1}$, $|D\omega|$, and $v$ are big
enough and equivalent with each other at $x_{0}$.

Substituting (\ref{e-11}), (\ref{e-12}), (\ref{e-13}) into
(\ref{e-10}) yields
\begin{eqnarray} \label{e-14}
I_{1}\geq
O(v^{3})+u_{1}\zeta^{2}\beta_{1}^{2}v+\sum\limits_{i=2}^{n}u_{1}\zeta^{2}\beta^{2}_{i}v^{3}\sigma^{ii}
+\sum\limits_{i=2}^{n}\left(O(v^{2})-\sigma^{ii}u_{1}\zeta\beta_{1}v^{2}\right)u_{ii}.
\end{eqnarray}
It is easy to observe that
\begin{eqnarray} \label{e-15}
I_{2}=O(v^{3}),\qquad I_{5}=O(v^{2}).
\end{eqnarray}
For the term $I_{4}$, we have
\begin{eqnarray}\label{e-16}
I_{4}&=&-4\sum\limits_{i,j,k=1}^{n}a^{ij}\sigma^{kk}u_{ki}G_{kj}\nonumber\\
&=&
-4u_{11}G_{11}-4(1+v^{2})\sum\limits_{i=2}^{n}\sigma^{ii}u_{1i}G_{1i}-4v^{2}\sum\limits_{i=2}^{n}(\sigma^{ii})^{2}u_{ii}G_{ii}\nonumber\\
&\geq&
-\left[2u_{11}^{2}+2G_{11}^{2}+\frac{1+v^{2}}{2}\sum\limits_{i=2}^{n}(\sigma^{ii})^{2}u^{2}_{1i}+8(1+v^{2})\sum\limits_{i=2}^{n}G^{2}_{1i}\right]\nonumber\\
&&\quad-\left[\frac{v^{2}}{2}\sum\limits_{i=2}^{n}(\sigma^{ii})^{2}
u_{ii}^{2}+8v^{2}\sum\limits_{i=2}^{n}(\sigma^{ii})^{2}G_{ii}^{2}\right].
\end{eqnarray}
Combing (\ref{e-3}), (\ref{e-4}), (\ref{e-15}) and (\ref{e-16}), one
can easily get
\begin{eqnarray}\label{e-17}
\sum\limits_{i=2}^{5}I_{i}&\geq&\frac{3}{2}(1+v^{2})\sum\limits_{i=2}^{n}(\sigma^{ii})^{2}(u_{1i})^{2}+\frac{3}{2}v^{2}\sum\limits_{i=2}^{n}(\sigma^{ii})^{2}(u_{ii})^{2}+O(v^{3})\nonumber\\
&=&
\frac{3}{2}(1+v^{2})\sum\limits_{i=2}^{n}(\sigma^{ii})^{2}\left(O(1)-\frac{\alpha\beta_{i}v}{2}-\frac{\omega^{i}}{v}u_{ii}\right)^{2}\nonumber\\
&&\quad +\frac{3}{2}v^{2}\sum\limits_{i=2}^{n}(\sigma^{ii})^{2}(u_{ii})^{2}+O(v^{3})\nonumber\\
&=&
O(v^{3})+\frac{3}{8}(1+v^{2})\sum\limits_{i=2}^{n}(\sigma^{ii})^{2}\alpha^{2}\beta^{2}_{i}v^{2}\nonumber\\
&&\quad+\sum\limits_{i=2}^{n}(\frac{3}{2}v^{2}(\sigma^{ii})^{2}+O(1))u^{2}_{ii}+\sum\limits_{i=2}^{n}O(v^{2})u_{ii}.
\end{eqnarray}
Applying (\ref{e-14}), (\ref{e-17}) and the fact that
$ax^{2}+bx\geq-\frac{b^{2}}{4a}$ for $a>0$, we can get
\begin{eqnarray}\label{e-18}
\sum\limits_{i=1}^{5}I_{i}&\geq&
O(v^{3})+u_{1}\zeta^{2}\beta^{2}_{1}v+\sum\limits_{i=2}^{n}u_{1}\sigma^{ii}\zeta^{2}\beta^{2}_{i}v^{3}
+\frac{3}{8}(1+v^{2})\sum\limits_{i=2}^{n}(\sigma^{ii})^{2}\zeta^{2}\beta^{2}_{i}v^{2}\nonumber\\
&&\quad+
\sum\limits_{i=2}^{n}(\frac{3}{2}v^{2}(\sigma^{ii})^{2}+O(1))u^{2}_{ii}
+\sum\limits_{i=2}^{n}(O(v^{2})-\sigma^{ii}u_{1}\alpha\beta_{1}v^{2})u_{ii}\nonumber\\
&\geq&
O(v^{3})+u_{1}\zeta^{2}\beta^{2}_{1}v+\sum\limits_{i=2}^{n}u_{1}\sigma^{ii}\zeta^{2}\beta^{2}_{i}v^{3}
+\frac{3}{8}(1+v^{2})\sum\limits_{i=2}^{n}(\sigma^{ii})^{2}\zeta^{2}\beta^{2}_{i}v^{2}\nonumber\\
&&\quad-
\sum\limits_{i=2}^{n}\frac{\left[O(v^{2})-\sigma^{ii}u_{1}\zeta
\beta_{1}v^{2}\right]^{2}}{6v^{2}(\sigma^{ii})^{2}+O(1)}.
\end{eqnarray}
Since $v$ has been assumed to be large enough, and $u_{1}$,
$\omega_{1}$, $\omega^{1}$, $|D\omega|$, and $v$ are equivalent with
each other, we have
\begin{eqnarray} \label{e-19}
\zeta^{2}\beta^{2}_{1}+\sum\limits_{i=2}^{n}u_{1}\sigma^{ii}\zeta^{2}\beta^{2}_{i}v
+\frac{3}{8}(1+v^{2})\sum\limits_{i=2}^{n}(\sigma^{ii})^{2}\zeta^{2}\beta^{2}_{i}\geq\zeta^{2}v^{2}\sum\limits_{i=2}^{n}\sigma^{ii}\beta^{2}_{i}
\end{eqnarray}
and
\begin{eqnarray} \label{e-20}
-\sum\limits_{i=2}^{n}\frac{\left[O(v^{2})-\sigma^{ii}u_{1}\zeta
\beta_{1}v^{2}\right]^{2}}{6v^{4}(\sigma^{ii})^{2}
+O(v^{2})}\geq-\frac{n-1}{5}\zeta^{2}v^{2}\beta^{2}_{1}+O(1).
\end{eqnarray}
Substituting (\ref{e-19}), (\ref{e-20}) into (\ref{e-18}), one can
easily get
\begin{eqnarray} \label{e-21}
I=\sum\limits_{i,j=1}^{n}\frac{a^{ij}|D\omega|^{2}_{ij}
}{|D\omega|^{2}}&=&\frac{\sum_{i=1}^{5}I_{i}}{|D\omega|^{2}}\nonumber\\
&\geq&O(v)+\alpha^{2}v^{2}\sum\limits_{i=2}^{n}\sigma^{ii}\beta^{2}_{i}-\frac{n-1}{5}\alpha^{2}v^{2}\beta^{2}_{1}.
\end{eqnarray}
Substituting (\ref{e-7}), (\ref{e-8}), (\ref{e-21}) into (\ref{e-2})
yields
\begin{eqnarray*}
0&\geq& a^{ij}\Phi_{ij}\\
&\geq&
O(v)+\zeta^{2}v^{2}\sum\limits_{i=2}^{n}\sigma^{ii}\beta^{2}_{i}-\frac{n-1}{5}\zeta^{2}v^{2}\beta^{2}_{1}\\
&&\quad+
\zeta k_{0}\left(1+(n-1)\eta+(n-1)\eta u^{2}_{1}\right)-\zeta^{2}\left(\beta^{2}_{1}+(1+u_{1}^{2})\sum\limits_{i=2}^{n}\sigma^{ii}\beta^{2}_{i}\right)\\
&\geq& O(v)-\frac{n-1}{5}\zeta^{2}v^{2}\beta^{2}_{1}+\zeta
k_{0}(n-1)\eta v^{2},
\end{eqnarray*}
which, by taking $0<\zeta<\min\{2\kappa_{1}, 5k_{0}\eta\}$, leads to
a contradiction. Hence, $Du$ must be bounded at $x_0$. Since
$\Omega$ is compact, an easy argument can give a universal constant
$c_1$, depending only on $n$, $L$ and $\Omega$ itself, as the upper
bound for $\max_{\Omega}|Du|$. This completes the proof of Lemma
\ref{lemma2.1}.
\end{proof}

Using Lemma \ref{lemma2.1} and a similar method to that in the proof
of \cite[Theorem 1.3]{mww} (actually, this technique has been shown
in \cite{aw1} already), we have:

\begin{proof}[Proof of Theorem \ref{maintheorem}] We are going to
use an approximation argument to get our main conclusion.

We first show that for any given $\varepsilon>0$ and $\upsilon\in
\mathbb{R}$, there exists a unique solution to the following BVP
  \begin{eqnarray*}
(\ast_{\varepsilon, \upsilon})\qquad \left\{
\begin{array}{ll}
 \mathrm{div}\left(\frac{Du}{\sqrt{ 1+|Du|^{2} }}\right)=\upsilon+\varepsilon u\qquad \qquad &\mathrm{in} ~ \Omega,\\
 D_{\vec{\nu}}u=\phi(x)\qquad \qquad &\mathrm{on} ~\partial\Omega.
 \end{array}
\right.
\end{eqnarray*}
For fixed $\varepsilon>0$, if $\upsilon=0$, using a similar argument
to that in the proof of \cite[Theorem 3.2]{ggm}, we can get the
$C^{0}$-estimate of the solution $u_{\varepsilon,0}$ to
$(\ast_{\varepsilon,0})$. Together with Lemma \ref{lemma2.1}, the
existence of the solution $u_{\varepsilon,0}$ follows. Besides, by
using Hopf's lemma, the uniqueness of the solution can be obtained
also.

Set
 \begin{eqnarray*}
 u_{\varepsilon,\upsilon}(x):=u_{\varepsilon,0}(x)-\frac{\upsilon}{\varepsilon}.
 \end{eqnarray*}
It is clear that $u_{\varepsilon,\upsilon}(x)$ is strictly
decreasing w.r.t. $\upsilon$. It is easy to check that
$u_{\varepsilon,\upsilon}(x)$ solves the BVP $(\ast_{\varepsilon,
\upsilon})$. Because of the construction of
$u_{\varepsilon,\upsilon}(x)$, its uniqueness is obvious.

We \textbf{claim}:
 \begin{itemize}

\item for any $\varepsilon>0$, there exists a unique, uniformly
bounded constant $\upsilon_{\varepsilon}$ such that
$|u_{\varepsilon,\upsilon_{\varepsilon}}(x)|_{C^{1}(\overline{\Omega})}$
is uniformly bounded.

\end{itemize}

Let $u_{0}(x)\in C^{\infty}(\overline{\Omega})$ be a fixed function
with $D_{\vec{\nu}}u_{0}=\phi(x)$. If one chooses
 \begin{eqnarray*}
 M:=1+\max_{\overline{\Omega}}|u_{0}|+\max_{\overline{\Omega}}\left|\mathrm{div}\left(\frac{Du_{0}}{\sqrt{1+|Du_{0}|^{2}}}\right)\right|
 \end{eqnarray*}
and
 \begin{eqnarray*}
u_{\varepsilon}^{+}:=u_{0}+\frac{M}{\varepsilon},
  \end{eqnarray*}
then, by linearization process, it follows that
\begin{eqnarray}\label{e-22}
0&<&M-\mathrm{div}\left(\frac{Du_{0}}{\sqrt{1+|Du_{0}|^{2}}}\right)+\varepsilon u_{0}\nonumber\\
 &=&\left[\mathrm{div}\left(\frac{Du_{\varepsilon,0}}{\sqrt{1+|Du_{\varepsilon,0}|^{2}}}\right)-\varepsilon u_{\varepsilon,0}\right]
-\left[\mathrm{div}\left(\frac{Du_{\varepsilon}^{+}}{\sqrt{1+|Du_{\varepsilon}^{+}|^{2}}}\right)-\varepsilon u_{\varepsilon}^{+}\right]\nonumber\\
&=&\sum\limits_{k,l=1}^{n}D_{k}\left[m_{kl}(x)D_{l}
(u_{\varepsilon,0}-u^{+}_{\varepsilon})\right]-\varepsilon(u_{\varepsilon,0}-u_{\varepsilon}^{+})
\end{eqnarray}
where
 \begin{eqnarray*}
  m_{kl}(x)=\int^{1}_{0} \frac{\partial
A^{k}}{\partial
p_{l}}(sDu_{\varepsilon,0}+(1-s)Du_{\varepsilon}^{+})ds
 \end{eqnarray*}
 and
  \begin{eqnarray*}
A^{k}(\vec{p})=\frac{p^{k}}{\sqrt{1+|\vec{p}|^{2}}}.
 \end{eqnarray*}
Hence, by applying the fact
\begin{eqnarray*}
D_{\vec{\nu}}(u_{\varepsilon,0}-u^{+}_{\varepsilon})=0
\end{eqnarray*}
and the maximum principle of second-order PDEs, we know that
$u^{+}_{\varepsilon}$ is a supersolution of
$(\ast_{\varepsilon,0})$. Similarly,
$u^{-}_{\varepsilon}:=u_{0}-\frac{M}{\varepsilon}$ is a subsolution
of $(\ast_{\varepsilon,0})$. Therefore, we have
$u_{\varepsilon,M}<u_{0}<u_{\varepsilon,-M}$.

Since $u_{\varepsilon,\upsilon}$ is strictly decreasing, for any
$\varepsilon\in(0,1)$, there exists a unique
$\upsilon_{\varepsilon}\in(-M,M)$ such that
$u_{\varepsilon,\upsilon_{\varepsilon}}(0)=u_{0}(0)$. By Lemma
\ref{lemma2.1}, one can easily get the uniform bound, which is
independent of $\varepsilon$, for
$|Du_{\varepsilon,\upsilon_{\varepsilon}}|$. Besides, since
$u_{\varepsilon,\upsilon_{\varepsilon}}(0)=u_{0}(0)$ and
$|Du_{\varepsilon,\upsilon_{\varepsilon}}|$ is uniformly bounded,
the uniform $C^0$-bound follows directly. Hence, we can get the
uniform bound for
$|u_{\varepsilon,\upsilon_{\varepsilon}}|_{C^{1}(\overline{\Omega})}$
(i.e., our \textbf{claim} is true), and by the Schauder theory of
second-order PDEs, the uniform higher order derivative estimates can
be ensured.

Letting $\varepsilon\rightarrow0$, extracting subsequence if
necessary, we can infer that there exists a constant $\lambda$ and a
smooth function $u^{\infty}(x)$ such that
 \begin{eqnarray*} \label{e-23}
\upsilon_{\varepsilon}+\varepsilon
u_{\varepsilon,\upsilon_{\varepsilon}}\rightarrow\lambda, \qquad
u_{\varepsilon,\upsilon_{\varepsilon}}\rightarrow u^{\infty},
\end{eqnarray*}
and it is easy to check that $(\lambda,u^{\infty})$ satisfies
$(\natural)$.

At the end, we would like to show the uniqueness of the solution to
$(\natural)$. If $(\chi,u^{\chi})$ also solves $(\natural)$, then
almost the same argument as in (\ref{e-22}) yields
 \begin{eqnarray} \label{e-23}
  \left\{
\begin{array}{ll}
 \sum\limits_{k,l=1}^{n}D_{k}\left[\widetilde{g}_{kl}(x)D_{l}(u^{\infty}-u^{\chi})\right]=\lambda-\chi\qquad \qquad &\mathrm{in} ~ \Omega,\\
 D_{\vec{\nu}}(u^{\infty}-u^{\chi})=0 &\mathrm{on}~\partial\Omega,
 \end{array}
\right.
 \end{eqnarray}
where $\{\widetilde{g}_{kl}(x)\}$ is positive definite and
 \begin{eqnarray*}
\widetilde{g}_{kl}(x)=\int^{1}_{0}\frac{\partial A^{k}}{\partial
p_{l}}(sDu^{\infty}+(1-s)Du^{\chi})ds.
 \end{eqnarray*}
According to (\ref{e-23}), integrating by parts gives that
$\lambda=\chi$ and then Hopf's lemma shows that
$u^{\infty}-u^{\chi}$ must be a constant. This finishes the proof of
Theorem \ref{maintheorem}.
\end{proof}

\vspace{5mm}

\section*{Acknowledgments}
\renewcommand{\thesection}{\arabic{section}}
\renewcommand{\theequation}{\thesection.\arabic{equation}}
\setcounter{equation}{0} \setcounter{maintheorem}{0}

This research was supported in part by the NSF of China (Grant Nos.
11801496 and 11926352), the Fok Ying-Tung Education Foundation
(China), and Hubei Key Laboratory of Applied Mathematics (Hubei
University).


\begin{thebibliography}{99}


\bibitem{ar} U. Abresch, H. Rosenberg, \emph{A Hopf differential for constant mean curvature surfaces in $\mathbb{S}^{2}\times\mathbb{R}$ and
$\mathbb{H}^{2}\times\mathbb{R}$}, Acta Math. {\bf 193} (2004)
141--174.


\bibitem{aw1} S.-J. Altschuler, L.-F. Wu, \emph{Translating surfaces of the non-parametric mean curvature flow with prescribed contact
angle}, Calc. Var. Partial Differential Equations {\bf 2} (1994)
101--111.

\bibitem{ggm} Y. Gao, Y.-J. Gong, J. Mao, \emph{Translating solutions of the nonparametric mean curvature flow with nonzero Neumann boundary data in product manifold
$M^{n}\times\mathbb{R}$}, available online at arXiv:2001.09860v2.

\bibitem{gm2} Y. Gao, J. Mao, \emph{Translating solutions of the nonparametric mean curvature flow with
nonzero Neumann boundary data in product manifold
$M^{n}\times\mathbb{R}$, II}, preprint.


\bibitem{mww} X.-N. Ma, P.-H. Wang, W. Wei, \emph{Constant mean curvature surfaces and mean curvature flow with nonzero Neumann boundary conditions on strictly
convex domains}, J. Funct. Anal. {\bf 274} (2018) 252--277.

\bibitem{lmh} L. Mazet, M.-M. Rodr\'{\i}guez, H. Rosenberg, \emph{Periodic constant mean curvature surfaces in
$\mathbb{H}^{2}\times\mathbb{R}$}, Asian J. Math. {\bf 18} (2014)
829--858.


\bibitem{mr} W.-H. Meeks III, H. Rosenberg, \emph{Stable minimal surfaces in
$M\times\mathbb{R}$},  J. Differential Geom. {\bf 68} (2004)
515--534.


\bibitem{hfj}H. Rosenberg, F. Schulze, J. Spruck, \emph{The half-space property
and entire positive minimal graphs in $M\times\mathbb{R}$}, J.
Differential Geom. {\bf 95} (2013) 321--336.


\bibitem{wwx} J. Wang, W. Wei, J.-J. Xu, \emph{Translating solutions of non-parametric mean curvature flows with capillary-type boundary value
problems}, Commun. Pure Appl. Anal. {\bf 18}(6) (2019) 3243--3265.




\end{thebibliography}
 \end{document}